\documentclass[a4paper,oneside,10pt]{amsart}

\usepackage[ngerman,english]{babel}
\usepackage{a4wide}
\usepackage[T1]{fontenc}
\usepackage[ansinew]{inputenc}
\usepackage{lmodern} 
\usepackage{graphicx}
\usepackage{amsmath}
\usepackage{amsthm}
\usepackage{amsfonts}
\usepackage{amssymb}
\usepackage{setspace}
\usepackage{mathrsfs}
\usepackage[all]{xy}
\usepackage{enumerate}

\usepackage{pgf,tikz}
\usetikzlibrary{arrows, automata}
\usepackage{capt-of}

\newtheorem{theorem}{Theorem}[section]

\newtheorem{proposition}[theorem]{Proposition}

\theoremstyle{definition}
\newtheorem{definition}[theorem]{Definition}

\newtheorem{remark}[theorem]{Remark}

\newtheorem*{theorem*}{Theorem}
\newtheorem*{proposition*}{Proposition}
\newtheorem*{definition*}{Definition}
\numberwithin{equation}{section}

\newenvironment{abc}{\begin{enumerate}[{\rm (a)}]}{\end{enumerate}}

\newenvironment{iiv}{\begin{enumerate}[{\rm (i)}]}{\end{enumerate}}

\def\G{\mathrm{G}}

\def\V{\mathrm{V}}
\def\vv{\mathrm{v}}
\def\ee{\mathrm{e}}
\def\dd{\mathrm{d}}
\def\E{\mathrm{E}}
\def\W{\mathrm{W}}
\def\dom{\mathrm{D}}
\def\NN{\mathbb{N}}
\def\CC{\mathbb{C}}
\def\BB{\mathbb{B}}
\def\RR{\mathbb{R}}
\def\Ell{\mathrm{L}}
\def\diag{\mathrm{diag}}
\def\Id{\mathrm{I}}
\def\LLL{\mathscr{L}}

\begin{document}
\title{Semigroups for flows on limits of graphs} 
\author{Christian Budde}
\address{North-West University, School of Mathematical and Statistical Sciences, Potchefstroom Campus, Private Bag X6001-209, Potchefstroom 2520, South Africa}
\email{christian.budde@nwu.ac.za}


\begin{abstract}
We use a version of the Trotter--Kato approximation theorem for strongly continuous semigroups in order to study flows on growing networks. For that reason we use the abstract notion of direct limits in the sense of category theory.
\end{abstract}

\date{}
\maketitle

\section*{Introduction}
Transport of goods is nowadays of extreme importance and indispensable considering what mankind needs for daily life. Now imagine a start-up company shipping special goods all over the world. Of course, the company starts with a small network of customers. However, assuming the company grows and retains the already existing routes and customers, the ship network grows and grows. It might come to the point in the development of the company, that one actually lost the view on all specific routes but only knows how the network works since it becomes too big. However, one still wants to know how the transport is going on the whole network.

\medskip
Mathematically speaking, the routes and customers can be described through edges and vertices of a graph, respectively. By giving the graph a metric structure one obtains so-called quantum graphs or networks. Transport can be modeled simply by the linearized flow equation
\[
\frac{\partial}{\partial{t}}w(t,x)=c\frac{\partial}{\partial{x}}w(t,x),
\]
where $c>0$ is the given velocity of the transported good. The combination of flows and networks in a pure operator algebraic setting was first worked out by M.~Kramar Fijav\v{z} and E.~Sikolya \cite{KS2005} for finite graphs. Later on, B.~Dorn generalized this also to infinite graphs \cite{Dorn2008}. The idea is to rewrite the transport equation on the network, which is subjected to some general boundary conditions due to the structure of the network, as an abstract Cauchy problem, which can be solved using $C_0$-semigroups. The techniques are also used frequently by other authors \cite{BF2017,BF2018,BDK2013} even for the situation one asks for weaker solutions \cite{BK2019} by means of bi-continuous semigroups. Notice that the motivating example given above is not the only one, actually, one can imagine a lot more scenarios, e.g., social networks \cite{A2015,AHMGPS2020} or brain connections \cite{SLKSJ2015}, just to name a few.

\medskip
That a network is growing, through adding vertices and edges, means that one has a sequence of at first finite graphs, i.e., one knows the exact structure as described in the situation above. In this case, each finite graph of the sequence yields a phase space $\Ell^1\left(\left[0,1\right],\CC^m\right)$, where $m\in\NN$ is the number of edges of the graph. It is important to notice, that hence the phase space changes depending on number of edges $m$. One assumes that each graph is a subgraph of the subsequent graph in the sequence, describing the above mentioned situation of growing networks. The above mentioned situation, where the network becomes to big in order to know all the network routes will be modelled by an infinite graph. Then we work on the Banach space $\Ell^1\left(\left[0,1\right],\ell^1\right)$. The transition from finite to infinite graphs will be modeled by direct limits in a certain category. The approximation of the transport process on the direct limit graph, is done by a version of the Trotter--Kato approximation theorem, which is originally due to T.~Kato \cite[Chapter IX, Thm.~3.6]{K1976} and H.F.~Trotter \cite[Thm.~5.2 \&~5.3]{T1958} and modified by a version by K.~Ito and F.~Kappel \cite[Thm.~2.1]{TK1998} which we are going to use. Actually this is related to the first Trotter--Kato theorem, cf. \cite[Chapter III, Thm.~4.8]{EN}. In the present paper we extend the work of Ito and Kappel by another approximation theorem, which is an extension of the second Trotter--Kato theorem, cf. \cite[Chapter III, Thm.~4.9]{EN}. We notice, that this paper deals with categorical limits of graphs. However, there is another notion of graph limits due to L.~Lov\'{a}sz \cite{Lov2012} by means of graphons or graphings. This notion of limits totally differs from what we consider within this paper but is worth to mention since this is a interesting topic which is supposed to yield future research on networks and dynamical systems. In fact, graph limits in the sense of L.~Lov\'{a}sz are in the focus of forthcoming papers.

\medskip
The structure of the paper is as follow: in the first section we recall all fundamentals on networks, flows on it and category theory. In Section \ref{sec:AppFlowDirLim} we apply the first Trotter--Kato theorem for our model of growing networks. The following section consists of the second Trotter--Kato approximation theorem in the style of Ito and Kappel.

\section{Preliminaries} 
\subsection{Graphs and networks}

In order to talk about finite and infinite networks we make use of the notation used in \cite{KS2005}, \cite{Dorn2008} or \cite{DKNR2010}. A network is modeled by a finite or infinite \emph{directed graphs} $G=(\V(G),\E(G))$, where $\V(G)=\left\{\vv_i:\ i\in I\right\}$ is the set of \emph{vertices} and $\E(G)=\left\{\ee_j:\ j\in J\right\}\subseteq \V\times\V$ is the set of \emph{directed edges} for some at most countable sets $I,J\subseteq\NN$. For a directed edge $\ee=(\vv_i,\vv_k)$, i$,k\in I$, we call $\vv_i$ the \emph{tail} and  $\vv_k$ the \emph{head} of $\ee$. Further, the edge $\ee$ is an \emph{outgoing edge} of the vertex $\vv_i$ and an \emph{incoming edge} for the vertex $\vv_k$. Recall that a graph $G$ is called \emph{simple} if there are neither loops nor multiple edges in $G$. This means in particular, that there are no edges of the form $\ee=(\vv_i,\vv_i)$, $i\in I$ (i.e., the tail and the head of the edge coincide and so an edge connects a vertex with itself) and no several edges connecting two vertices in the same direction. We also assume that the graph $\G$ is \emph{uniformly locally finite} meaning that each vertex has only finitely many outgoing edges and that the number of outgoing edges is uniformly bounded from above.  

\medskip 
The structure of a graph can be described by its incidence or its adjacency matrix. The \emph{outgoing incidence matrix} $\Phi^-=(\Phi_{ij}^-)$ is defined by
\begin{equation}
\Phi_{ij}^-:=\begin{cases}
1& \text{ if }\vv_i\stackrel{\ee_j}{\longrightarrow}.\\
0 & \text{otherwise},
\end{cases}
\end{equation}
By $\vv_i\stackrel{\ee_j}{\longrightarrow}$ we mean that the vertex $\vv_i$ is the tail of the edge $\ee_j$. The \emph{incoming incidence matrix} $\Phi^+=(\Phi_{ij}^+)$ is defined by 
\begin{equation}
\Phi_{ij}^+:=\begin{cases}
1& \text{ if }\stackrel{\ee_j}{\longrightarrow}\vv_i,\\
0 & \text{ otherwise}.
\end{cases}
\end{equation}
Here $\stackrel{\ee_j}{\longrightarrow}\vv_i$ means that the vertex $\vv_i$ is the head of the edge $\ee_j$. The \emph{incidence matrix} $\Phi$ of the directed graph $G$, describing the structure of the network completely, is then defined by $\Phi:=\Phi^+-\Phi^-$. There are two other important matrices associated to a general graph and which are needed in what follows. The \emph{transposed adjacency matrix} of the graph $G$ is defined by
\[
\mathbb{A}:=\Phi^+\left(\Phi^-\right)^{\top}.
\]
The nonzero entries of $\mathbb{A}$ correspond exactly to the edges of the graph, cf. \cite[p.~280]{Positive2017}. In fact, $\mathbb{A}$ can be described explicitly as
\begin{equation}
\mathbb{A}_{ij}:=\begin{cases}
1& \text{ if }\vv_j\stackrel{\ee_k}{\longrightarrow}\vv_i,\\
0 & \text{ otherwise}.
\end{cases}
\end{equation}
Last but not least we use the so-called \emph{(transposed) adjacency matrix of the line graph} $\mathbb{B} = (\mathbb{B}_{ij})$ defined by $\BB:=\left(\Phi^-\right)^\top\Phi^+$. One can also give an explicit entrywise description as
\begin{equation}\label{eqn:adjMat}
\mathbb{B}_{ij}:=\begin{cases}
1& \text{ if }\stackrel{\ee_j}{\longrightarrow}\vv_k\stackrel{\ee_i}{\longrightarrow},\\
0 & \text{ otherwise}.
\end{cases}
\end{equation}
Notice, that by the assumption of the uniform locally finiteness of the graph the matrix $\BB$ is a bounded operator on $\ell^1:=\ell^1(J)$. 

\medskip
In what follows, we stick to the following mathematical setting. We identify every edge of our graph with the unit interval,  $\ee_j\equiv\left[0,1\right]$ for each $j\in J$, and parametrize it contrary to the direction of the flow, if $c_j>0$, $j\in J$, so that it is assumed to have its tail at the endpoint $1$ and its head at the endpoint $0$, i.e., the material flows from $1$ to $0$. With this assumption, we stay within the framework introduced by B.~Dorn, M.~Kramar Fijav\v{z} and E.~Sikolya, see for example \cite{KS2005,Dorn2008}. For simplicity we use the notation $\ee_j(1)$ and $\ee_j(0)$ for the tail and the head, respectively. In this way we obtain a \emph{metric graph}. 

\subsection{Category theory}
By taking all simple locally finite directed graphs together, one obtains a rich mathematical structure by means of a category. We recall the most important definitions here as they can be found for example in the monographs by S.~Mac Lane \cite{CatML} or S.~Awodey \cite{CatA}. We first recap the basic definition of a category.

\begin{definition}
A category $\mathcal{C}$ consists of \emph{objects} $A,B,C,\ldots$ and \emph{arrows} $f,g,h,\ldots$ (also called \emph{morphisms}). For each arrow $f$ there are given objects $\mathrm{dom}(f)$ and $\mathrm{cod}(f)$ called the \emph{domain} and \emph{codomain} of $f$. We write $f: A \to B$ to indicate that $A = \mathrm{dom}(f)$ and $B = \mathrm{cod}(f)$. Given arrows $f : A \to B$ and $g : B \to C$, that is, with $\mathrm{cod}(f) = \mathrm{dom}(g)$ there is given an arrow $g \circ f : A \to C$ called the \emph{composite} of $f$ and $g$. Furthermore, for each object $A$ there is given an arrow $1_A : A \to A$ called the \emph{identity arrow} of $A$. These arrows are required to satisfy the following axiomas:
\begin{abc} 
	\item \emph{Associativity}, i.e., for $f:A\to B$, $g:B\to C$ and $h:C\to D$ one has 
	\[
	h\circ(g\circ f)=(h\circ g)\circ f,
	\]
	\item \emph{Unit law}, i.e., for each $f:A\to B$ one has 
	\[
	f\circ 1_A=f=1_B\circ f.
	\]
\end{abc}	
\end{definition}

As said above, the objects we are interested in are simple and locally finite graphs. In order to form a category, we need to specify what the arrows in the category are. For that reason we recall the definition of the so-called graph homomorphisms.

\begin{definition}
A \emph{(graph)-homomorphism} between two graphs $G=(\V(G),\E(G))$ and $H=(\V(H),\E(H))$ is a map $\varphi:\V(G)\to\V(H)$ such that $(\vv_1,\vv_2)\in\E(G)$ implies that $(\varphi(\vv_1),\varphi(\vv_2))\in\E(H)$. If such an homomorphism is injective, then $G$ is a \emph{subgraph} of $H$.
\end{definition}

By taking together graphs and their homomorphisms we obtain a category.

\begin{definition}
The category $\mathcal{C}:=\textbf{\text{SimpLocFinG}}$ consists of simple and locally finite graphs as objects and graph homomorphisms as arrows.
\end{definition}

In category theory constructions on categories, e.g., products of categories or free categories, as well as universals and limits play a central role. For the purpose of this paper we recall the following definition of a direct limit in a category. Notice that we simplified the original definition to sequences of objects instead of directed systems of objects, cf. \cite[Chapter V, Sect.~1]{CatML} or \cite[Def.~5.17 \& 5.18]{CatA}, in order to fit in our framework.

\begin{definition}
Let $\mathcal{C}$ be a category and $(A_n)_{n\in\NN}$ a sequence of objects in $\mathcal{C}$ such that there exist maps $\varphi_n:A_n\to A_{n+1}$ for each $n\in\NN$, i.e., we have the following diagram
\[
A_1\stackrel{\varphi_1}{\longrightarrow}A_2\stackrel{\varphi_2}{\longrightarrow}A_3\stackrel{\varphi_3}{\longrightarrow}A_4\longrightarrow\cdots
\]
We say that an objects $A$ in $\mathcal{C}$ is the \emph{direct limit} of the sequence $(A_n)_{n\in\NN}$ if for any $n\in\NN$ there exists an arrow $\psi_n:A_n\to A$ such that $\psi_{n+1}\circ\varphi_n=\psi_n$ for each $n\in\NN$, i.e., the following diagrams commute for each $n\in\NN$:
\[
\xymatrix{
A_n\ar[rd]_{\psi_n}\ar[r]^{\varphi_n}&A_{n+1}\ar[d]^{\psi_{n+1}}\\
& A
}
\]
Moreover, $A$ is \emph{universal} in the sense that if another object $B$ such that there exist arrows $\vartheta_n:A_n\to B$ such that $\vartheta_{n+1}\circ\varphi_n=\vartheta_n$ for all $n\in\NN$, then there exists a unique arrow $\alpha:A\to B$ such that $\alpha\circ\psi_n=\vartheta_n$ for each $n\in\NN$.
\end{definition}

We will use the concept of direct limits for our special category $\mathcal{C}=\textbf{\text{SimpLocFinG}}$ for the special case that the arrows $(\varphi_n)_{n\in\NN}$ are all injective, i.e., that we consider a sequence $(G_n)_{n\in\NN}$ of simple and locally finite graphs which is growing in the sense that $G_n$ is a subgraph of $G_{n+1}$ for each $n\in\NN$. As a matter of fact, each element in the sequence $(G_n)_{n\in\NN}$ is supposed to be a finite simple graph and hence locally finite. The direct limit itself does not have to be finite anymore.

\subsection{Flows in networks}
So far we considered the algebraic components of this paper. We will now turn to the analytical structure, i.e., we discuss transport processes on networks. Notice that we only treat the infinite graph case here, since the finite one is included and only needs some small modifications. Let $G=(\V(G),\E(G))$ be a simple and locally finite graph, then we study the following partial differential equation on the graph $G$ for $i\in I$, $j\in J$:
\begin{align*}\tag{PDE}\label{eqn:PDE}
\begin{cases}
\displaystyle{\frac{\partial}{\partial{t}}w_j(t,x)=c_j\frac{\partial}{\partial{x}}w_j(t,x)},&\quad x\in\left(0,1\right),\ t\geq0,\\
\displaystyle{w_j(x,0)=f_j(x)},&\quad x\in\left(0,1\right),\\
\displaystyle{\Phi^{-}_{ij}w_j(1,t)=\sum_{k\in\NN}{\Phi^{+}_{ik}w_k(0,t)}},&\quad t\geq0.
\end{cases}
\end{align*}
How this equation relates to the classical linear Boltzmann equation is described in \cite[Sect.~1]{DKNR2010}. In what follows, we assume that all velocities $c_j$, $j\in J$ on the edges stay away from zero and are bounded from above, i.e., there exist $m,M>0$ such that
\begin{align}\label{eqn:VelocitiesBound}
m\leq c_j\leq M,\quad j\in J.
\end{align}
The boundary conditions of the equation depend on the structure of the network which is introduced by the incidence matrices. Now consider the Banach space $X:=\Ell^1\left(\left[0,1\right],\ell^1\right)$ equipped with the norm given by
\[
\left\|f\right\|:=\int_0^1{\left\|f(s)\right\|_{\ell^1}\ \dd{s}}
\]
and introduce the (unbounded) operator $(A,\dom(A))$ on $X$ defined by
\begin{align}\label{eqn:DefAOp}
A:=\diag\left(\frac{\dd}{\dd{x}}\right),\quad \dom(A):=\left\{f\in\W^{1,1}\left(\left[0,1\right],\ell^1\right):\ f(1)=\BB_Cf(0)\right\},
\end{align}
where 
\begin{align}\label{eq:BC-Op}
\BB_C:=C^{-1}\BB C\quad \text{and}\quad C:=\diag\left(c_j\right).
\end{align}
Notice, that \eqref{eqn:VelocitiesBound} assures, that the operator $\BB_C$ is bounded. It is well-known that the corresponding abstract Cauchy problem given by
\begin{align*}\tag{ACP}\label{eqn:ACP}
\begin{cases}
\dot{u}(t)=Au(t),&\quad t\geq0\\
u(0)=f,
\end{cases}
\end{align*}
on the Banach space $X=\Ell^1\left(\left[0,1\right],\ell^1\right)$ is equivalent to the partial differential equation \eqref{eqn:PDE}, i.e., a solution of \eqref{eqn:ACP} gives rise to a solution of \eqref{eqn:PDE} and vice versa, cf. \cite[Prop.~3.1]{Dorn2008} and \cite{KS2005} for the finite graph case. We now need the notions of well-posedness of abstract Cauchy problems and $C_0$-semigroups, cf. \cite[Chapter II, Thm.~6.7]{EN}.

\begin{definition}
A function $u:\RR_+\to X$ is called a \emph{(classical) solution} of \eqref{eqn:ACP} if $u$ is continuously differentiable with respect to $X$, $u(t)\in\dom(A)$ for all $t\geq0$ and \eqref{eqn:ACP} holds.
\end{definition}

\begin{definition}\label{def:WP}
The abstract Cauchy problem \eqref{eqn:ACP} is called \emph{well-posed} if for every $f\in\dom(A)$, there exists a unique solution $u(\cdot,f)$ of \eqref{eqn:ACP}, $\dom(A)$ is dense in $X$ and if for every sequence $(f_n)_{n\in\NN}$ in $\dom(A)$ with $\lim_{n\to\infty}{f_n}=0$, one has $\lim_{n\to\infty}{u(t,f_n)}=0$ uniformly in compact intervals.
\end{definition}

\begin{definition}
A family of bounded linear operators $(T(t))_{t\geq0}$ is called \emph{strongly continuous one-parameter semigroup of linear operators}, or \emph{$C_0$-semigroup}, if the following properties are satisfied:
\begin{iiv}
	\item $T(t+s)=T(t)T(s)$ and $T(0)=\Id$ for all $t,s\geq0$.
	\item $\lim_{t\searrow0}{\left\|T(t)f-f\right\|}=0$ for each $f\in X$.
\end{iiv}
\end{definition}

Each $C_0$-semigroup $(T(t))_{t\geq0}$ gives rise to an operator $(A,\dom(A))$ called the generator. This operator is defined as follows.
\[
Ax:=\lim_{t\searrow0}{\frac{T(t)x-x}{t}},\quad \dom(A):=\left\{x\in X:\ \lim_{t\searrow0}{\frac{T(t)x-x}{t}}\ \text{exists}\right\}.
\]
The converse question, which operator $(A,\dom(A))$ is the generator of a $C_0$-semigroups is more involving. As a matter of fact, this question is answered by the so-called Hille--Yosida theorem, cf. \cite[Chapter II, Thm.~3.8]{EN}, \cite{Y1948}. If a given operator $(A,\dom(A))$ generates a $C_0$-semigroup $(T(t))_{t\geq0}$ on a Banach space $X$ satisfying $\left\|T(t)\right\|\leq M\ee^{\omega t}$ for some $M\geq1$, $\omega\in\RR$ and for all $t\geq0$, then we will denote this by $A\in\mathcal{G}(M,\omega,X)$. The most important fact is, that by \cite[Chapter II, Cor.~6.9]{EN} the abstract Cauchy problem \eqref{eqn:ACP} is well-posed in the sense of Definition \ref{def:WP} if and only if the operator $(A,\dom(A))$ is the generator of a $C_0$-semigroup. The following result shows, that our explicit abstract Cauchy problem \eqref{eqn:ACP} associated to \eqref{eqn:PDE} is well-posed, cf. \cite[Thm.~3.4]{Dorn2008}. For the finite network case, we refer to \cite[Prop.~2.5]{KS2005}.

\begin{theorem}\label{thm:WellPosedTranProb}
The operator $(A,\dom(A))$ on $X=\Ell^1\left(\left[0,1\right],\ell^1\right)$ defined by \eqref{eqn:DefAOp}, generates a $C_0$-semigroup. Therefore, \eqref{eqn:ACP} is well-posed.
\end{theorem}

\section{Approximation of flows on direct limit graphs}\label{sec:AppFlowDirLim}

We now consider the situation as described earlier. Let $(G_n=(\V(G_n),\E(G_n)))_{n\in\NN}$ be a growing sequence of finite simple graphs, i.e., there exist injective graph homomorphisms $\varphi_n:G_n\to G_{n+1}$ for each $n\in\NN$. Notice that by \cite[Def.~8.1]{P1995} the limit of such a sequence $(G_n)_{n\in\NN}$ exists. Let us denote this limit by $G=(\V(G),\E(G))$. In particular, one has $G=\bigcup_{n\in\NN}{G_n}$. For each $n\in\NN$ we have a strongly continuous semigroup $(T_n(t))_{t\geq0}$ solving \eqref{eqn:ACP} on the space $\Ell^1\left(\left[0,1\right],\CC\right)^{\left|\E(G_n)\right|}$. Moreover, we have a $C_0$-semigroup $(T(t))_{t\geq0}$ on $\Ell^1\left(\left[0,1\right],\ell^1\right)$. The clue is, that the semigroups $(T_n(t))_{t\geq0}$ approximate the semigroup $(T(t))_{t\geq0}$ is a certain sense. To make this more precise, we refer to the work of K.~Ito and F.~Kappel \cite{TK1998}. In fact, we will use the following theorem.

\begin{theorem}{\cite[Thm.~2.1]{TK1998}}\label{thm:TK}
Let $X$ and $X_n$, $n\in\NN$, be Banach spaces and such that for each $n\in\NN$ there exist bounded linear operators $P_n:X\to X_n$ and $E_n:X_n\to X$ such that $\sup_{n\in\NN}{\left\|P_n\right\|}<\infty$, $\sup_{n\in\NN}{\left\|E_n\right\|}<\infty$ and $P_nE_n=\Id_n$, where $\Id_n$ denotes the identity operator on $X_n$, $n\in\NN$. Let $A\in\mathcal{G}(M,\omega,X)$ and $A_n\in\mathcal{G}(M,\omega,X_n)$ for each $n\in\NN$ and let $(T(t))_{t\geq0}$ and $(T_n(t))_{t\geq0}$ be the semigroups generated by $A$ and $A_n$ on $X$ and $X_n$, respectively. Then the following statements are equivalent.
\begin{abc}
	\item There exists $\lambda_0\in\rho(A)\cap\bigcap_{n\in\NN}{\rho(A_n)}$, such that for all $x\in X$,
	\[
	\lim_{n\to\infty}{\left\|E_nR(\lambda_0,A_n)P_nx-R(\lambda_0,A)x\right\|}=0.
	\]
	\item For every $x\in X$ and $t\geq0$
	\[
	\lim_{n\to\infty}{\left\|E_nT_n(t)P_nx-T(t)x\right\|}=0
	\]
	uniformly on bounded $t$-intervals.
\end{abc}
\end{theorem}

In order to apply this theorem we have to specify all required data for our situation. The choices for the Banach spaces are clear, i.e., one chooses 
\[
X_n:=\Ell^1\left(\left[0,1\right],\CC\right)^{\left|\E(G_n)\right|}=\Ell^1\left(\left[0,1\right],\CC^{\left|\E(G_n)\right|}\right),\quad n\in\NN,
\]
and
\[
X:=\Ell^1\left(\left[0,1\right],\ell^1\right).
\]
Now we have to specify what the operators $(P_n)_{n\in\NN}$ and $(E_n)_{n\in\NN}$ have to be in our case. Since $G=(\V(G),\E(G))$ is the direct limit of the sequence $(G_n)_{n\in\NN}$ there exist graph homomorphisms $\psi_n:G_n\to G$, $n\in\NN$, which by \cite[Prop.~8.3]{P1995} are injective, too. These maps yield maps $E_n:X_n\to X$ for $n\in\NN$. Actually, the operator $E_n$ intuitively extend the functions on $G$ by infinitely many zeros. Without loss of generality, we may assume that the labeling of the edges of $G_n$ and $G$ coincide on $\E(G)\setminus\E(G_n)$. To be more detailed, the operator $E_n$ has the following action
\[
E_n(f_1,f_2,\ldots,f_{\left|E(G_n)\right|})=(f_1,f_2,\ldots,f_{\left|\E(G_n)\right|},0,0,0\ldots).
\] 
The operators $P_n:X\to X_n$, $n\in\NN$, are just the restrictions to the smaller subspace, i.e., it is a cut-off operator. Again, by assuming that labeling of the edges of $G_n$ and $G$ coincide on $\E(G)\setminus\E(G_n)$ we can describe $P_n$ as follows
\[
P_n(f_1,f_2,f_3,\ldots)=(f_1,f_2,\ldots,f_{\left|\E(G_n)\right|})
\]
By construction it is clear that $\left\|P_n\right\|\leq1$ and $\left\|E_n\right\|\leq1$ for each $n\in\NN$.

\begin{remark}\label{rem:FunLim}
The sequence of (injective) graph homomorphisms $(\varphi_n)_{n\in\NN}$ corresponding to the sequence $(G_n)_{n\in\NN}$ of graphs extends, similar to the maps $(\psi_n)_{n\in\NN}$, the a sequence of $(\Phi_n)_{n\in\NN}$ of linear maps $\Phi_n:X_n\to X_{n+1}$. It is easy to verify that $\Ell^1\left(\left[0,1\right],\ell^1\right)$ is the direct limit of the Banach spaces $(X_n)_{n\in\NN}$ in the category of Banach spaces with contractions as morphisms. More details regarding functors and categories of Banach spaces can for example been found in the monograph by P.W.~Michor \cite{M1978}. 
\end{remark}

Let us set the observation of Remark \ref{rem:FunLim} in a bigger picture by means of category theory. For that reason, we recall the following definition, cf. \cite[Chapter I, Sect.~3]{CatML} or \cite[Def.~1.2]{CatA}.

\begin{definition}
Let $\mathcal{C}_1$ and $\mathcal{C}_2$ be two categories. A \emph{functor} $F:\mathcal{C}_1\to\mathcal{C}_2$ between the categories $\mathcal{C}_1$ and $\mathcal{C}_2$ is a mapping of objects to objects and arrows to arrows, such that
\begin{iiv}
	\item $F(f:A\to B)=F(f):F(A)\to F(B)$,
	\item $F(1_A)=1_{F(A)}$,
	\item $F(f\circ g)=F(f)\circ F(g)$.
\end{iiv}
\end{definition}

In other words, a functor $F$ preserves domains and codomains, identity arrows, and composition. Now let us apply the concept of functors for our situation. In particular, let us denote the category of Banach spaces together with linear bounded operator between them as arrows by $\mathcal{B}$. Then there exists a functor $F:\mathcal{C}\to\mathcal{B}$ by $F(G=(\V(G),\E(G)))=\Ell^1\left(\left[0,1\right],\CC\right)^{\left|\E(G)\right|}$ if $G$ is finite and $F(G=(\V(G),\E(G)))=\Ell^1\left(\left[0,1\right],\ell^1\right)$ if $G$ is infinite. Moreover, for the graph homomorphism $\varphi_n:G_n\to G_{n+1}$ one defines $F(\varphi_n)=\Phi_n$.

\medskip
Let us now come back to our transport problem. By the previous section, we know that the operator $A$ and $A_n$ are in fact generators of $C_0$-semigroups $(T(t))_{t\geq0}$ and $(T_n(t))_{t\geq0}$, respectively. The following result shows, that the semigroups $(T_n(t))_{t\geq0}$ \glqq \emph{converge}\grqq{} to $(T(t))_{t\geq0}$ in the sense of Theorem \ref{thm:TK}$\mathrm{(b)}$. 

\begin{proposition}
Let $(G_n)_{n\in\NN}$ be a increasing sequence of graphs with limit $G$. By $((A_n\dom(A_n))_{n\in\NN}$ and $(A,\dom(A))$ we denote the operators defined by \eqref{eqn:DefAOp} associated to the transport problems on $X_n$ and $X$, respectively. Then, for every $x\in X$ and $t\geq0$ one has that ${\left\|E_nT_n(t)P_nx-T(t)x\right\|}\to0$ for $n\to\infty$ uniformly on bounded $t$-intervals. Intuitively spoken, the semigroups $(T_n(t))_{t\geq0}$, $n\in\NN$, approximate $(T(t))_{t\geq0}$ along the growing sequence of graphs.
\end{proposition}

\begin{proof}
In order to prove the result, we make use of Theorem \ref{thm:TK}. By \cite[Prop.~18.12]{Positive2017} one has that $\lambda\in\rho(A)$ if $\mathrm{Re}(\lambda)>0$ and for such a $\lambda\in\rho(A)$ one has
\[
R(\lambda,A)=\left(\Id+E_{\lambda}(\cdot)(1-\BB_{C,\lambda})^{-1}\BB_{C,\lambda}\otimes\delta_0\right)R_{\lambda},
\]
where $\delta_0$ denotes the point evaluation at $0$, $E_\lambda(s):=\diag\left(\ee^{(\lambda/c_j)s}\right)$, $\BB_{C,\lambda}:=E_{\lambda}(-1)\BB_C$, see also \eqref{eq:BC-Op}, and
\[
(R_\lambda f)(s):=\int_s^1{E_{\lambda}(s-t)C^{-1}f(t)\ \dd{t}}.
\]
Hence, it is clear that $\rho(A)\cap\bigcup_{n\in\NN}{\rho(A_n)}\neq\varnothing$. By the explicit description of the operator $(E_n)_{n\in\NN}$, $(P_n)_{n\in\NN}$ and the resolvents it is clear that ${\left\|E_nR(\lambda_0,A_n)P_nx-R(\lambda_0,A)x\right\|}\to0$ for $n\to\infty$. Therefore, by Theorem \ref{thm:TK} we conclude that for every $x\in X$ and $t\geq0$ one has that ${\left\|E_nT_n(t)P_nx-T(t)x\right\|}\to0$ for $n\to\infty$ uniformly on bounded $t$-intervals.
\end{proof}

\begin{remark}
We mentioned in the introduction, that the theory for flows in networks has been generalized by M.~Kramar Fijav\v{z} and the author in \cite{BK2019} to a bigger class of operator semigroups on the phase space $\Ell^{\infty}\left(\left[0,1\right],\ell^1\right)$, the so-called bi-continuous semigroups. These objects have a rich structure and have been introduced by F.~K\"uhnemund \cite{KuPhD} and further developed by B.~Farkas and the author \cite{BF,BF2,BPos}. We will not go into the details of this theory since this is not the topic of this paper. Nevertheless, it is worth to mention that even in the case of bi-continuous semigroups, there are Trotter--Kato approximation theorems \cite{AK2002,AM2004} in the spirit of \cite[Chapter III, Thm.~4.8 \& 4.9]{EN}. Moreover, there is a recent paper on the first Trotter--Kato theorem which is closely related to the work of Ito and Kappel, cf. \cite{TKBi2019}. We would like to notice, that even if there are these approximation theorems for bi-continuous semigroups, the procedure described above is not imitable. This is due to the fact, that one has to assume that the network is finite if one allows velocities on the edges of the network which are not rational (and linear dependent). Unfortunately, this is due to the absence of a Lumer--Phillips type generation theorem for bi-continuous semigroups. However, if one assumes that even the direct limit is finite, then the procedure from above just works out.
\end{remark}

\section{A second Trotter--Kato type theorem}
The assumption in Theorem \ref{thm:TK} is that we know that there exists a $C_0$-semigroup on the spaces $X$ and in some sense we know how to approximate them by means of resolvents. However, another important question is, if there exists a $C_0$-semigroup as a limit on $X$ if I only know that there exists a sequence of semigroups on the spaces $X_n$. The following theorem is related to the second Trotter--Kato theorem on a single Banach space \cite[Chapter III, Thm.~4.9]{EN}. We now formulate this theorem such that it fits into the framework of K.~Ito and F.~Kappel.

\begin{theorem}\label{thm:TK2}
Let $X$ and $X_n$, $n\in\NN$, be Banach spaces and such that for each $n\in\NN$ there exist bounded linear operators $P_n:X\to X_n$ and $E_n:X_n\to X$ such that $\sup_{n\in\NN}{\left\|P_n\right\|}<\infty$, $\sup_{n\in\NN}{\left\|E_n\right\|}<\infty$ and $P_nE_n=\Id_n$, where $\Id_n$ the the identity operator on $X_n$, $n\in\NN$. Let $A_n\in\mathcal{G}(M,\omega,X_n)$ for each $n\in\NN$ and let $(T_n(t))_{t\geq0}$, $n\in\NN$, be the semigroups generated by $A_n$ on $X_n$. Then the following statements are equivalent.
\begin{abc}
	\item There exists $\lambda_0\in\rho(A)\cap\bigcap_{n\in\NN}{\rho(A_n)}$ and a bounded operator $R\in\LLL(X)$ with dense range, such that for all $x\in X$,
	\[
	\lim_{n\to\infty}{\left\|E_nR(\lambda_0,A_n)P_nx-Rx\right\|}=0.
	\]
	\item There exists a strongly continuous semigroup $(T(t))_{t\geq0}$ on $X$ such that for every $x\in X$ and $t\geq0$
	\[
	\lim_{n\to\infty}{\left\|E_nT_n(t)P_nx-T(t)x\right\|}=0
	\]
	uniformly on bounded $t$-intervals.
\end{abc}
\end{theorem}

\begin{proof}
The implication $\mathrm{(b)}\Rightarrow\mathrm{(a)}$ is just an application of Theorem \ref{thm:TK}. For the converse, assume that the assertion $\mathrm{(a)}$ holds. First of all, we notice that $\left\{R(\lambda):\ \mathrm{Re}(\lambda)>0\right\}$ with
\[
R(\lambda)x:=\lim_{n\to\infty}{E_nR(\lambda,A_n)P_nx},\quad x\in X,
\]
is a pseudoresolvent on $X$ such that
\begin{align}\label{eqn:PseudoResEst}
\left\|\lambda^kR(\lambda)^k\right\|\leq M,\quad k\in\NN.
\end{align}
Recall from \cite[Chapter III, Def.~4.3]{EN} that $\left\{R(\lambda):\ \mathrm{Re}(\lambda)>0\right\}$ is a pseudoresolvent if $R(\lambda)$ is a bounded linear operator for all $\lambda>0$ with $\mathrm{Re}(\lambda)>0$ and the equality
\[
R(\lambda)-R(\mu)=(\lambda-\mu)R(\lambda)R(\mu),
\]
holds for all $\lambda,\mu\in\CC$ with $\mathrm{Re}(\lambda)>0$ and $\mathrm{Re}(\mu)>0$. 

\medskip
To see, that $\left\{R(\lambda):\ \mathrm{Re}(\lambda)>0\right\}$ is indeed a pseudoresolvent, we slightly modify \cite[Chapter III, Prop.~4.4]{EN}. In fact, consider the set 
\[
\Gamma:=\left\{\lambda\in\CC:\ \mathrm{Re}(\lambda)>0,\ \lim_{n\to\infty}{E_nR(\lambda,A_n)P_nx}\ \text{exists for all}\ x\in X\right\}.
\]
By the assumptions of assertion (a) we have that $\Gamma\neq\varnothing$. By \cite[Chapter IV, Prop.~1.3]{EN}, one has that for a given $\mu\in\Gamma$
\[
E_nR(\lambda,A_n)P_n=\sum_{k\in\NN}{(\mu-\lambda)^kE_nR(\mu,A_n)^{k+1}P_n},
\]
whenever $\left|\mu-\lambda\right|<\mathrm{Re}(\mu)$, where the convergence is with respect to the operator norm and uniform in $\left\{\lambda\in\CC:\ \left|\mu-\lambda\right|<\alpha\mathrm{Re}(\mu)\right\}$ for each $\alpha\in\left(0,1\right)$. By the fact that $\sup_{n\in\NN}\left\|E_n\right\|<\infty$ and $\sup_{n\in\NN}\left\|P_n\right\|<\infty$ we see that $E_nR(\lambda,A_n)P_nx$ converges for all $\lambda$ satisfying $\left|\mu-\lambda\right|<\alpha\mathrm{Re}(\mu)$ whenever $n\to\infty$. We conclude, that $\Gamma$ is an open set in $\CC_+:=\left\{\lambda\in\CC:\ \mathrm{Re}(\lambda)>0\right\}$. On the other hand side, let $\lambda$ with $\mathrm{Re}(\lambda)>0$ be an accumulation point of $\Gamma$. For $\alpha\in\left(0,1\right)$ one can find $\mu\in\Gamma$ such that $\left|\mu-\lambda\right|<\alpha\mathrm{Re}(\mu)$. By what we have seen before, $\lambda\in\Gamma$ showing that $\Gamma$ is also closed in $\CC_+$. Since $\CC_+$ is a connected space, we conclude that the only subsets which are both open and closed are $\varnothing$ and $\CC_+$. Since we observed that $\Gamma\neq\varnothing$, we have $\Gamma=\CC_+$.

\medskip
Finally, we have to show that \eqref{eqn:PseudoResEst} holds. To do so, we observe that we have the following estimate for $\lambda>0$ due to the assumption that $A_n\in\mathcal{G}(M,\omega,X_n)$ for each $n\in\NN$
\[
\left\|\lambda R(\lambda)\right\|\leq\lim_{n\to\infty}\left\|E_n\right\|\cdot\left\|\lambda R(\lambda,A_n)\right\|\cdot\left\|P_n\right\|\leq M_1\left\|R(\lambda,A_n)\right\|\leq M_1M_2,
\]
where $M_1:=\sup_{n\in\NN}\left\|E_n\right\|\cdot\sup_{n\in\NN}{\left\|P_n\right\|}\geq0$ and $M_2\geq0$ a constant such that $\left\|\lambda R(\lambda,A_n)\right\|\leq M_2$ which exists since $A_n\in\mathcal{G}(M,\omega,X_n)$ for each $n\in\NN$ as an application of the Hille--Yosida generation theorem for strongly continuous semigroups, cf. \cite[Chapter II, Thm.~3.8]{EN}. This finally leads to the fact that \eqref{eqn:PseudoResEst} is satisfied.

\medskip                                                       
Since by construction $\mathrm{Ran}(R(\lambda))=\mathrm{Ran}(R)$, which is dense by assumption, we conclude that there exists a densely defined operator $(B,\dom(B))$ on $X$ such that $R(\lambda)=R(\lambda,B)$ for $\lambda>0$, cf. \cite[Chapter III, Cor.~4.7]{EN}. Hence, the operator $(B,\dom(B))$ satisfies the following estimate 
\[
\left\|\lambda^kR(\lambda,B)^k\right\|\leq M,\quad k\in\NN,
\]
yielding a bounded strongly continuous semigroup $(T(t))_{t\geq0}$ on $X$. By a second application of Theorem \ref{thm:TK} we conclude that assertion $\mathrm{(b)}$ has to be true.
\end{proof}

\begin{remark}
Notice that Theorem \ref{thm:TK} allows to approximate a given semigroup by other semigroups. However, the assertion of Theorem \ref{thm:TK2} is stronger in the sense, that one has not to know that there exists a semigroup on the space $X$ but that one has approximants which behaves well in the sense that they eventually converge to one semigroup.
\end{remark}

\section{Example}
In this final section we consider an example of a growing sequence $(G_n)_{n\in\NN}$ of networks. We only show the first two elements $G_1$ and $G_2$ of the sequence since it has an obvious pattern.

\begin{minipage}[t]{0.5\textwidth}
\begin{center}
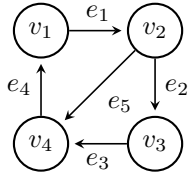

\begin{tikzpicture}[
            > = stealth, 
            shorten > = 2pt, 
            auto,
            node distance = 1.5cm, 
            semithick 
        ]

        \tikzstyle{every state}=[
            draw = black,
            thick,
            fill = white,
            minimum size = 4mm
        ]

        \node[state] (v1) {$v_1$};
        \node[state] (v2) [right of=v1] {$v_2$};
        \node[state] (v3) [below of=v2] {$v_3$};
        \node[state] (v4) [below of=v1] {$v_4$};
        
        \path[->] (v1) edge node {$e_1$} (v2);
        \path[->] (v2) edge node {$e_2$} (v3);
				\path[->] (v3) edge node {$e_3$} (v4);
				\path[->] (v4) edge node {$e_4$} (v1);
				\path[->] (v2) edge node {$e_5$} (v4);
\end{tikzpicture}
\captionof{figure}{Graph of $G_1$}
\end{center}
\end{minipage}
\begin{minipage}[t]{0.5\textwidth}
\begin{center}
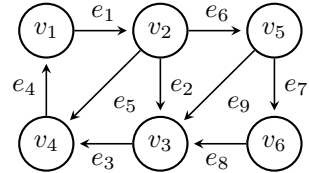

\begin{tikzpicture}[
            > = stealth, 
            shorten > = 2pt, 
            auto,
            node distance = 1.5cm, 
            semithick 
        ]

        \tikzstyle{every state}=[
            draw = black,
            thick,
            fill = white,
            minimum size = 4mm
        ]

        \node[state] (v1) {$v_1$};
        \node[state] (v2) [right of=v1] {$v_2$};
        \node[state] (v3) [below of=v2] {$v_3$};
        \node[state] (v4) [below of=v1] {$v_4$};
				\node[state] (v5) [right of=v2] {$v_5$};
				\node[state] (v6) [right of=v3] {$v_6$};
        
        \path[->] (v1) edge node {$e_1$} (v2);
        \path[->] (v2) edge node {$e_2$} (v3);
				\path[->] (v3) edge node {$e_3$} (v4);
				\path[->] (v4) edge node {$e_4$} (v1);
				\path[->] (v2) edge node {$e_5$} (v4);
				
				\path[->] (v2) edge node {$e_6$} (v5);
				\path[->] (v5) edge node {$e_7$} (v6);
				\path[->] (v6) edge node {$e_8$} (v3);
				\path[->] (v5) edge node {$e_9$} (v3);
				
\end{tikzpicture}
\captionof{figure}{Graph of $G_2$}
\end{center}
\end{minipage}
We assume that the velocities are all equal to $1$. The corresponding weighted (transposed) adjacency matrix of the line graph $\mathbb{B}_1$ and $\mathbb{B}_2$ of the graphs $G_1$ and $G_2$ are given by
\[
\mathbb{B}_1=\begin{pmatrix}
0&0&0&1&0\\
1&0&0&0&0\\
0&1&0&0&0\\
0&0&1&0&1\\
1&0&0&0&0
\end{pmatrix}
\]
and
\[
\mathbb{B}_2=\begin{pmatrix}
0&0&0&1&0&0&0&0&0\\
1&0&0&0&0&0&0&0&0\\
0&1&0&0&0&0&0&1&1\\
0&0&1&0&1&0&0&0&0\\
1&0&0&0&0&0&0&0&0\\
1&0&0&0&0&0&0&0&0\\
0&0&0&0&0&1&0&0&0\\
0&0&0&0&0&0&1&0&0\\
0&0&0&0&0&1&0&0&0\\
\end{pmatrix}
=
\left(\begin{array}{@{}c|c@{}}
\mathbb{B}_1 & 
\begin{matrix}
0&0&0&0\\
0&0&0&0\\
0&0&1&1\\
0&0&0&0\\
0&0&0&0
\end{matrix}\\
\hline
\begin{matrix}
1&0&0&0&0\\
0&0&0&0&0\\
0&0&0&0&0\\
0&0&0&0&0\\
\end{matrix} &
\begin{matrix}
0&0&0&0\\
1&0&0&0\\
0&1&0&0\\
1&0&0&0
\end{matrix}
\end{array}\right)
\]
Since we assumed that the velocities are all equal to one, we can make use of the explicit expression of the transport semigroups $(T_1(t))_{t\geq0}$ and $(T_2(t))_{t\geq0}$ on the networks $G_1$ and $G_2$, respectively, cf. \cite[Sect.~3]{Dorn2008} or \cite[Sect.~18.2]{BKFR2017}. In particular, one has
\[
T_i(t)f(s)=\mathbb{B}_i^kf(x+t-k),
\] 
for $k\in\NN$ such that $k\leq t+x<k+1$ and $f\in\Ell^1\left(\left[0,1\right],\CC^{\left|\E(G_i)\right|}\right)$ where $i=1,2$. By this, one observes that $(T_1(t))_{t\geq0}$ is indeed a restriction of $(T_2(t))_{t\geq0}$.

\section*{Acknowledgement}
This article is based upon work from COST Action CA18232 MAT-DYN-NET, supported by COST (European Cooperation in Science and Technology). The author is very grateful to the anonymous referee suggestions. They helped to improve the article. Especially, the style of the article overall has been improved.

\end{document}